\documentclass{amsart}
\usepackage{amsfonts}
\usepackage{graphicx}
\usepackage{url}
\usepackage{amssymb,amsmath,amsthm}
\usepackage{amsaddr, amssymb, epsfig}
\usepackage[active]{srcltx}

\usepackage{enumerate}

\usepackage[all]{xy}
\input xy
\xyoption{all}
\usepackage{xy}

\theoremstyle{plain}

\newtheorem{theorem}{Theorem}

\newtheorem{lemma}[theorem]{Lemma}

\newtheorem{corollary}[theorem]{Corollary}

\newtheorem{example}[theorem]{Example}

\usepackage{color}
\usepackage{mathrsfs}         
\usepackage[all]{xy}
\usepackage[mathscr]{euscript}
\let\euscr\mathscr \let\mathscr\relax
\usepackage[scr]{rsfso}

\newcommand{\Aut}[0]{\ensuremath{\mathsf{Aut}}}

\newcommand{\Inv}[0]{\ensuremath{\mathsf{Inv}}}
\newcommand{\shE}[0]{\ensuremath{\mathsf{shE}}}

\newcommand{\NP}[0]{\ensuremath{\mathsf{NP}}}

\newcommand{\coNP}[0]{\ensuremath{\mathsf{co\mbox{-}NP}}}

\newcommand{\Logspace}[0]{\ensuremath{\mathsf{Logspace}}}

\newcommand{\Pspace}[0]{\ensuremath{\mathsf{Pspace}}}

\newcommand{\FO}[0]{\ensuremath{\mathsf{FO}}}

\newcommand{\mylogic}{\ensuremath{\{\exists, \forall, \wedge,\vee \} \mbox{-}\mathsf{FO}}}
\newcommand{\posFO}{\ensuremath{\{\exists, \forall, \wedge,\vee,= \} \mbox{-}\mathsf{FO}}}
\newcommand{\tuple}[1]{\ensuremath{\mathbf{#1}}}
\newcommand{\notsubseteq}{\ensuremath{ \subseteq \hspace{-3mm} / \hspace{2mm} }}

\newcommand{\she}[3]{
\resizebox{!}{.4cm}{
\ensuremath{
\begin{array}{c|c}
1 & #1 \\
\hline
2 & #2 \\
\hline
3 & #3 
\end{array}
}
}
}

\newcommand{\shefour}[4]{
\resizebox{!}{.5cm}{
\ensuremath{
\begin{array}{c|c}
1 & #1 \\
\hline
2 & #2 \\
\hline
3 & #3 \\
\hline
4 & #4 
\end{array}
}
}
}

\newcommand{\shefive}[5]{
\resizebox{!}{.6cm}{
\ensuremath{
\begin{array}{c|c}
1 & #1 \\
\hline
2 & #2 \\
\hline
3 & #3 \\
\hline
4 & #4 \\
\hline
5 & #5 
\end{array}
}
}
}

\newcommand{\shee}[2]{
\resizebox{!}{.4cm}{
\ensuremath{
\begin{array}{c|c}
1 & #1 \\
\hline
2 & #2 \\
\end{array}
}
}
}

\begin{document}

%
%

\title{The lattice and semigroup structure of multipermutations}
\author{Catarina Carvalho}

\address{Department of Physics, Astronomy and Mathematics\\ University of Hertfordshire, College Lane, Hatfield AL10 9AB}
\email{c.carvalho2@herts.ac.uk}

\author{Barnaby Martin}

\address{Department of Computer Science, Durham University\\
Science Labs, South Road, Durham, DH1 3LE, UK}

\email{ barnabymartin@gmail.com}


\maketitle

\begin{abstract}
We study the algebraic properties of binary relations whose underlying digraph is smooth, that is has no source or sink. Such objects have been studied as surjective hyper-operations (shops) on the corresponding vertex set, and as binary relations that are defined everywhere and whose inverse is also defined everywhere. In the latter formulation, they have been called multipermutations. 

We study the lattice structure of sets (monoids) of multipermutations over an $n$-element domain. Through a Galois connection, these monoids form the algebraic counterparts to sets of relations closed under definability in positive first-order logic without equality. The first side of this Galois connection has been elaborated previously, we show the other side here. We study the property of inverse on multipermutations and how it connects our monoids to groups. We use our results to give a simple dichotomy theorem
for the evaluation problem of positive first-order logic without equality on the class of structures whose preserving multipermutations form a monoid closed under inverse. These problems turn out either to be in Logspace or to be
Pspace-complete. We go on to study the monoid of all multipermutations on an $n$-element domain,  under usual composition of relations. We characterise its Green relations, regular elements and show that it does not admit a generating set  that is polynomial on $n$. 
\end{abstract}


\section{Introduction}

A \emph{multipermutation} is a binary relation $\phi$ over a set $[n]=\{1,\ldots,n\}$ so that: for all $x \in [n]$ exists $y \in [n]$ such that $\phi(x,y)$; and for all $y \in [n]$ exists $x \in [n]$ such that $\phi(x,y)$. The term multipermutation originates with Schein in \cite{schein1987multigroups} but they were studied independently as \emph{surjective hyper-operations} (shops) in \cite{MadelaineM12,Martin10,MadelaineM18}.

In Universal Algebra there is a family of Galois connections that links relational expressivity in fragments of first-order logic with closure operators that generate particular types of algebra. The most important of these links expressivity in the fragment of first-order logic containing $\{\exists,\wedge,=\}$ (called \emph{primitive positive}, or pp-logic) with superpositional closure of finite arity operations in objects (called \emph{clones}) containing the projections. A survey of these Galois connections can be found in \cite{BornerPoschelSushchansky} (see Table 1) and a survey more oriented towards Computer Scientists containing much of the same material is \cite{BornerBasics} (see Tables 1 and 2).

The model-checking problem for primitive positive logic on a fixed relational structure $\mathcal{B}$ is known as the \emph{Constraint Satisfaction Probem} CSP$(\mathcal{B})$. The relevant Galois connection just noted gave rise to the so-called \emph{algebraic approach} to the computational complexity of CSP$(\mathcal{B})$. This approach culminated in the proof of the Feder-Vardi Conjecture that such problems for finite $\mathcal{B}$ are either in P or are NP-complete \cite{BulatovFVConjecture,ZhukFVConjecture}. One side of the Galois connection we discuss in this paper played a similar role in resolving the computational complexities of the corresponding model-checking problems for the fragment of first-order logic containing $\{\forall,\exists,\wedge,\vee\}$ \cite{MadelaineM18}. A complete Galois connection has two sides and it is the other side of this connection that we prove in this paper.

Galois connections have been leveraged in a variety of contexts related to the CSP in order to aid in classifications of computational complexity, where they have played or continue to play a key part in those projects. In the case of the \emph{Quantified CSP}, the relevant connection is noted in \cite{BBCJK} and involves surjective operations that preserve the corresponding relations (known in this case as well as the non-surjective case as \emph{polymorphisms}). The complexity classification for \emph{Quantified CSP} is famously wide open \cite{ZhukM20}. Another relative of the CSP, where the complexity classification is now known, is the \emph{Valued CSP}. Here the corresponding ``logic'' is no longer a fragment of first-order logic. The corresponding Galois connection was discovered gradually, culminating in the notion of \emph{weighted clones} in \cite{CohenCCJZ13}. The algebraic approach was pivotal in the final complexity classification for Valued CSPs \cite{ThapperZ16,KolmogorovKR17}, though the full power of weighted clones turned out not to be necessary (the more restricted notion of \emph{fractional polymorphism} was enough). A final relative of the CSP,  where the complexity classification is now still open, is the \emph{Promise CSP}. Again, the corresponding ``logic'' is not a fragment of first-order logic. The Galois connection here first appeared in \cite{Pip02} and was used subsequently in \cite{BrakensiekG18}. The algebraic approach here is ongoing and, indeed, promising (see \cite {BulinKO19}).

The Galois connection that we prove in this paper, related with the complexity of  $\mylogic(\mathcal{B})$, deals with 
sets of multipermutations $\phi$ over the set $B$ under which the relational $\mathcal{B}$ (whose domain is $B$) is \emph{invariant}, in the following sense. For a $k$-ary relation $R$ of $\mathcal{B}$ and $(x_1,\ldots,x_k) \in R$, always $(y_1,\ldots,y_k) \in R$, if $(x_1,y_1),\ldots,(x_k,y_k) \in \phi$. These sets, always containing the identity and closed under composition and subrelations (that are themselves also multipermutations), are known as \emph{down shop-monoids} (DSMs). In this paper we study the structure of the lattice of DSMs on a $k$-element domain. We give a full Galois connection proving an isomorphism between the lattice of DSMs (over a $k$-element domain) and the lattice of $k$-element structures closed under definabilty in positive first-order logic without equality. One half of this Galois connection was given in \cite{MadelaineM12}, the other half appears here for the first time. We study in particular the automorphism born of the inverse operation on multipermutations. DSMs that are closed under inverse have a fundamentally group-like structure -- what we call \emph{blurred permutation subgroups} (BPSs). Using this characterisation, we prove a dichotomy for our evaluation problem on structures that we term \emph{she-complementative}, i.e. whose monoid of permutations under which they are invariant, is closed under inverse. Specifically, these problems are either in \Logspace\ or are \Pspace-complete. This complexity classification follows from the general result of \cite{MadelaineM18} but our proof here is a great deal simpler.

Multipermutations have been studied earlier as monoids of binary relations, when considered closed only under composition of relations (not also subrelations). Schein \cite{schein1987multigroups} looked at sets $\Phi$ of binary relations $\phi$ defined everywhere (i.e. where every element of the domain appears in the first component) that are closed under inverses (i.e. $\phi^{-1}\in \Phi$ for all $\phi \in \Phi$), so both the domain and range of these relations are the full domain. Having termed these objects multipermutations, he went on to characterise involutive semigroups of multipermutations. Furthermore, he proved that every involutive semigroup of difunctional multipermutations is an inverse semigroup, and every inverse semigroup
 is isomorphic to an involutive semigroup of difunctional multipermutations.  Ten years later McKenzie and Schein \cite{mckenzie1997every}, after showing that every  semigroup is isomorphic to a transitive semigroup of binary relations, leave as an open problem the question ``Which semigroups are isomorphic to transitive semigroups of multipermutations?''  As far as we know this question is still open.

Bredikhin \cite{bredikhin1993representations} studied the monoid of all difunctional multipermutations on a $k$-element domain. Here the operation considered was not the usual composition of operations, since the composition of two difunctional relations is not necessarily difunctional. His idea on studying these monoids seemed to be to present a unification of the theories of inverse semigroups and lattices, see also \cite{bredikhin1996can}. 
These are, as far as we are aware, the only  articles mentioning  multipermutations. With their reappearance in the context mentioned above, the research for existing structural results did not throw much up, so we aim here to start filling this gap by studying the monoid of all multipermutations on a $k$-element domain.  

The monoid of (all) binary relations on a $k$-element domain has been widely studied since the 60s, as have some of its subsemigroups like the full transformation monoid, Hall monoid and, more recently, diagram semigroups. The fact that binary relations can also be represented as boolean square matrices and as graphs allows us to use techniques from different areas of mathematics to  study these monoids.
Drawing on similarities with previously studied monoids, we look at some structural properties of the monoid of multipermutations.
We characterise Green's relations, give an algorithm to compute regular elements, and show that this monoid, unlike the symmetric group,  does not admit a generating set that is of size polynomial in $k$. Finally we prove that blurred permutations are the completely regular difunctional multipermutations. 

There are still many questions to answer about the monoid of all multipermutations. What are its maximal subgroups? Is this semigroup better behaved in any way than the semigroup of all binary relations?  And a range of questions that can be posed about semigroups of multipermutations. For instance, are these semigroups, in any way,  better behaved than semigroups of binary relations?

This paper  is, partially, based on \cite{Martin10}. Some of the content from \cite{Martin10} is now obsolete and has been removed, other parts appear here (with minor issues) corrected. The section on the monoid of multipermutations, Section~\ref{sec:Mn}, is new to this paper.

\medskip
\noindent \textbf{Presentation}.
The paper is organised as follows. In Section~\ref{sec:prelims} we give the necessary preliminaries and introduce the Galois connection. In Section~\ref{sec:structure}, we discuss the structure of our lattices, with particular emphasis on an automorphism born of an inverse operation. We go on to prove the characterisation theorem that allows us to derive the complexity dichotomy for she-complementative structures. In Section~\ref{sec:Mn} we study the monoid of all  multipermutations on a $k$-element domain.

\section{Preliminaries}

\vspace{-.2cm}
\label{sec:prelims}
Let $\mathcal{B}$ be a finite structure, with domain $B$, over an at most countable relational signature $\sigma$. Let \mylogic\ be the positive fragment of first-order (fo) logic without equality. An \emph{extensional} relation is one that appears in the signature $\sigma$. We will usually denote extensional relations of $\mathcal{B}$ by $R$ and other relations by $S$ (or by some formula that defines them). In \mylogic\, the atomic formulae are exactly substitution instances of extensional relations. The problem $\mylogic(\mathcal{B})$ has:
\begin{itemize}
\item Input: a sentence $\varphi \in \mylogic$.
\item Question: does $\mathcal{B} \models \varphi?$
\end{itemize}
QCSP$(\mathcal{B})$ is the restriction of this problem to formulae involving no disjunction, what in our notation would be $\{ \exists, \forall, \wedge \}$-\FO. When $\mathcal{B}$ is of size one, the evaluation of any \FO\ sentence may be accomplished in \Logspace\ (essentially, the quantifiers are irrelevant and the problem amounts to the \emph{boolean sentence value problem}, see \cite{NLynch}). In this case, it follows that $\mylogic(\mathcal{B})$ is in \Logspace. Furthermore, by inward evaluation of the quantifiers, $\mylogic(\mathcal{B})$ is readily seen to always be in \Pspace.  

For a structure $\mathcal{B}$ define the complement structure $\overline{\mathcal{B}}$ to be over the same domain $B$ with relations which are the set-theoretic complements of those of $\mathcal{B}$. That is, for each $r$-ary $R$, $R^{\overline{\mathcal{B}}} = B^r \setminus R^{\mathcal{B}}$. Similarly, for a relation $R\subseteq B^r$, let $\overline{R}$ denote $B^r \setminus R$.

Consider the finite set $X=[n]:=\{1,\ldots,n\}$ and its power set $\euscr{P}(X)$. A \emph{hyper-operation} on $X$ is a function $f:X \rightarrow \euscr{P}(X) \setminus \{\emptyset\}$ (that the image may not be the empty set corresponds to the hyper-operation being \emph{total}, in the parlance of \cite{BornerTotalMultifunctions}). If the hyper-operation $f$ has the additional property that
\begin{itemize}
\item for all $y \in X$, there exists $x \in X$ such that $y \in f(x)$,
\end{itemize}
then we designate (somewhat abusing terminology) $f$ \emph{surjective}. A surjective hyper-operation in which each element is mapped to a singleton set is identified with a \emph{permutation} (bijection). 
Instead of operations we can think of these hyper-operations as being binary relations,  so that $f\subseteq X\times X$  satisfies
$$\forall x\in X \ \exists y_1, y_2\in X \ {\rm s.t.} \ (x, y_1), (y_2, x) \in f.$$ 
following the work of Schein \cite{Schein}, we denote  these {\it multipermutations}. We keep this terminology for both  the functions and relations.

A \emph{surjective hyper-endomorphism} (she) of a set of relations (forming the finite-domain structure) $\mathcal{B}$ over $X$ is a multipermutation $f$ on $X$ that satisfies, for all relations $R$ of $\mathcal{B}$,
\begin{itemize}
\item if $(x_1,\ldots,x_i) \in R$ then, for all $y_1\in f(x_1),\ldots,y_i \in f(x_i)$, $(y_1,\ldots,y_i) \in R$.
\end{itemize}
More generally, for $r_1,\ldots,r_k \in X$, we say $f$ is \emph{a she from} $(\mathcal{B};r_1,\ldots,r_k)$ to $(\mathcal{B};r'_1,\ldots,r'_k)$ if $f$ is a she of $\mathcal{B}$ and $r'_1 \in f(r_1), \ldots, r'_k \in f(r_k)$.
A she may be identified with a \emph{surjective endomorphism} if each element is mapped to a singleton set. On finite structures surjective endomorphisms are necessarily automorphisms.

\subsection{Galois Connections}

\subsubsection{Relational side.}

For a set $F$ of multipermutations on the finite domain $B$, let $\Inv(F)$ be the set of relations on $B$ of which each $f \in F$ is a she (when these relations are viewed as a structure over $B$). We say that $S \in \Inv(F)$ is invariant or is \emph{preserved} by (the multipermutations in) $F$. Let $\shE(\mathcal{B})$ be the set of shes of $\mathcal{B}$. Let $\Aut(\mathcal{B})$ be the set of automorphisms of $\mathcal{B}$.

Let $\langle \mathcal{B} \rangle_{\mylogic}$ and $\langle \mathcal{B} \rangle_{\posFO}$ be the sets of relations that may be defined on $\mathcal{B}$ in \mylogic\ and \posFO, respectively.

\begin{lemma}[\cite{LICS2009}]
\label{lemma:galois-connection-by-types}
Let $\tuple{r}:=(r_1,\ldots,r_k)$ be a $k$-tuple of elements of the finite-signature $\mathcal{B}$. There exists:
\begin{itemize}
\item[$(i).$] a formula $\theta_\tuple{r}(u_1,\ldots,u_k) \in \posFO$ s.t. $(\mathcal{B}, r'_1,\ldots, r'_k) \models \theta_\tuple{r}(u_1,\ldots,u_k)$ iff there is an automorphism from $(\mathcal{B}, r_1,\ldots, r_k)$ to $(\mathcal{B}, r'_1,\ldots, r'_k)$.
\item[$(ii).$] a formula $\theta_\tuple{r}(u_1,\ldots,u_k) \in \mylogic$ s.t. $(\mathcal{B}, r'_1,\ldots, r'_k) \models \theta_\tuple{r}(u_1,\ldots,u_k)$ iff there is a she from $(\mathcal{B}, r_1,\ldots, r_k)$ to $(\mathcal{B}, r'_1,\ldots, r'_k)$.
\end{itemize}
\end{lemma}

\noindent The following is the main theorem of \cite{LICS2009}.
\begin{theorem}[\cite{LICS2009}]
\label{thm:galois-connection}
For a finite-signature structure $\mathcal{B}$ we have
\begin{itemize}
\item[$(i).$] $\langle \mathcal{B} \rangle_{\posFO} = \Inv(\Aut(\mathcal{B}))$ and
\item[$(ii).$] $\langle \mathcal{B} \rangle_{\mylogic} = \Inv(\shE(\mathcal{B}))$.
\end{itemize}
\end{theorem}
We will need a countable-signature version of this theorem for our final lattice isomorphism.
\begin{theorem}
\label{thm:galois-connection-inf}
For a countable-signature structure $\mathcal{B}$ we have
\begin{itemize}
\item[$(i).$] $\langle \mathcal{B} \rangle_{\posFO} = \Inv(\Aut(\mathcal{B}))$ and
\item[$(ii).$] $\langle \mathcal{B} \rangle_{\mylogic} = \Inv(\shE(\mathcal{B}))$.
\end{itemize}
\end{theorem}
\begin{proof}
Again, Part $(i)$ is well-known and may be proved in a similar, but simpler, manner to Part $(ii)$, which we now prove. The direction $[\varphi(\tuple{v}) \in \langle \mathcal{B} \rangle_{\mylogic} \ $ $ \Rightarrow \ \varphi(\tuple{v}) \in \Inv(\shE(\mathcal{B}))]$ is proved as before.

For $[S \in \Inv(\shE(\mathcal{B})) \ \Rightarrow \ S \in \langle \mathcal{B} \rangle_{\mylogic}]$, we proceed similarly to before, but using finiteness of the domain $B$, which will rescue us from the pitfalls of an infinite signature. Consider the finite disjunction we previously built:
\[ \theta_S(u_1,\ldots,u_k) \ := \ \theta_{\tuple{r}_1}(u_1,\ldots,u_k) \vee \ldots \vee \theta_{\tuple{r}_m}(u_1,\ldots,u_k). \]
\noindent Let $R_1,R_2,\ldots$ be an enumeration of the extensional relations of $\mathcal{B}$. Let $\mathcal{B}_i$ be the reduct of $\mathcal{B}$ to the signature $\langle R_1,\ldots,R_i\rangle$. For $j \in [m]$ let $\theta^i_{\tuple{r}_j}(u_1,\ldots,u_k)$ be built as in Lemma~\ref{lemma:galois-connection-by-types}, but on the reduct $\mathcal{B}_i$. The relations $\theta^1_{\tuple{r}_j}(u_1,\ldots,u_k)$, $\theta^2_{\tuple{r}_j}(u_1,\ldots,u_k)$, \ldots are monotone decreasing on $B^k$ -- the shes must preserve an increasing number of extensional relations -- and therefore reach a limit $l_j$ s.t. $\theta^{l_j}_{\tuple{r}_j}(u_1,\ldots,u_k)$ =$\theta_{\tuple{r}_j}(u_1,\ldots,u_k)$. Let $l:=\max\{l_1,\ldots,l_m\}$ and build $\theta_S(u_1,\ldots,u_k)$ over the finite-signature reduct $\mathcal{B}_l$. The result follows.
\end{proof}
In the following, $\leq_{\Logspace}$ indicates the existence of a logspace many-to-one reduction.
\begin{theorem}
\label{thm:she-reduction}
Let $\mathcal{B}$ and $\mathcal{B}'$ be structures over the same domain $B$ s.t. $\mathcal{B}'$ is finite-signature.
\begin{itemize}
\item[$(i).$] If $\Aut(\mathcal{B}) \subseteq \Aut(\mathcal{B}')$ then $\posFO(\mathcal{B}') \leq_{\Logspace} \posFO(\mathcal{B})$.
\item[$(ii).$] If $\shE(\mathcal{B}) \subseteq \shE(\mathcal{B}')$ then $\mylogic(\mathcal{B}') \leq_{\Logspace} \mylogic(\mathcal{B})$.
\end{itemize}
\end{theorem}
\begin{proof}
Again, Part $(i)$ is well-known and the proof is similar to that of Part $(ii)$, which we give. If $\shE(\mathcal{B}) \subseteq \shE(\mathcal{B}')$, then $\Inv(\shE(\mathcal{B}')) \subseteq \Inv(\shE(\mathcal{B}))$. From Theorem~\ref{thm:galois-connection}, it follows that $\langle \mathcal{B}' \rangle_{\mylogic}$ $\subseteq \langle \mathcal{B} \rangle_{\mylogic}$. Recalling that $\mathcal{B}'$ contains only a finite number of extensional relations, we may therefore effect a logspace reduction from $\mylogic(\mathcal{B}')$ to $\mylogic(\mathcal{B})$ by straightforward substitution of predicates.
\end{proof}

\subsubsection{Down-shop-monoids and the functional side.}

Consider the finite domain $X$. The \emph{identity} multipermutation $id_X$ is defined by $x \mapsto \{x\}$. Given multipermutations $f$ and $g$, define the \emph{composition} $g \circ f$ by $x \mapsto \{ z : \exists y \ z \in g(y) \wedge y \in f(x) \}$. Finally, a multipermutation $f$ is a \emph{sub-multipermutation} of $g$ -- denoted $f \subseteq g$ -- if $f(x) \subseteq g(x)$, for all $x$.  
A set of multipermutations on a finite set $B$ is a \emph{down-shop-monoid} (DSM), if it contains $id_B$, and is closed under composition and sub-multipermutations\footnote{Closure under sub-multipermutations is termed \emph{down closure} in \cite{BornerTotalMultifunctions}, hence the D in DSM.} (of course, not all sub-hyper-operations of a multipermutation are surjective -- we are only concerned with those that are). $id_B$ is a she of all structures with domain $B$, and, if $f$ and $g$ are shes of $\mathcal{B}$, then so is $g \circ f$. Further, if $g$ is a she of $\mathcal{B}$, then so is $f$ for all (surjective) $f \subseteq g$. It follows that $\shE(\mathcal{B})$ is always a DSM. If $F$ is a set of permutations, then we write $\langle F\rangle_G$ to denote the group generated by $F$. If $F$ is a set of multipermutations on $B$, then let $\langle F \rangle_{DSM}$ denote the minimal DSM containing the operations of $F$.  If $F$ is the singleton $\{f\}$,  then, by abuse of notation, we write $\langle f \rangle$ instead of $\langle \{f\} \rangle$. 
We will mark-up, e.g., the multipermutation $1 \mapsto \{1,2\}$, $2\mapsto \{2\}$, $3\mapsto \{1,3\}$ as $\she{12}{2}{13}$.

For a multipermutation $f$, define its inverse $f^{-1}$ by $x \mapsto \{y : x \in f(y)\}$. Note that $f^{-1}$ is also a multipermutation and $(f^{-1})^{-1}=f$, though $f \circ f^{-1}=id_B$ only if $f$ is a permutation. For a set of multipermutation $F$, let $F^{-1}:=\{f^{-1}:f \in F\}$.

A \emph{permutation subgroup} on a finite set $B$ is a set of permutations of $B$ closed under composition. It may easily be verified that such a set contains the identity and is closed under inverse. A permutation subgroup may be identified with a particular type of DSM in which all multipermutations have only singleton sets in their range.

\begin{theorem}
\label{thm:functional-galois}
Let $F$ be a set of permutations (Part $(i)$) or multipermutations (Part $(ii)$) on the finite domain $X$. Then
\begin{itemize} 
\item[$(i)$] $\langle F\rangle_{G} = \Aut(\Inv(F))$, and
\item[$(ii)$] $\langle F\rangle_{DSM} = \shE(\Inv(F))$.
\end{itemize}
\end{theorem}
\begin{proof}

Part $(i)$ is well-known but we give a proof for illustrative purposes.

[$\langle F\rangle_{G} \subseteq \Aut(\Inv(F))$.] By induction. One may easily see that if $f,g \in \Aut(\Inv(F))$ then $f \circ g \in \Aut(\Inv(F))$. Further, if $f \in \Aut(\Inv(F))$ then $f^{-1} \in \Aut(\Inv(F))$ as the set of automorphisms is closed under inverse.

[$\Aut(\Inv(F)) \subseteq \langle F\rangle_{G}$.] Let $|X|=n$. One may easily see that $\Inv(F)=\Inv(\langle F\rangle_G)$ (for the forward containment, note that inverse follows from the fact that $F$ is a set of bijections on a finite set). Let $R$ be the $n$-ary relation that lists the permutations in $\langle F\rangle_{G}$ (e.g., the identity appears as $(1,2,\ldots,n)$); $R$ is preserved by $\langle F\rangle_G$. We will prove $\Aut(\Inv(\langle F\rangle_G)) \subseteq \langle F\rangle_{G}$ by contraposition. If $g$ is a permutation not in $\langle F\rangle_{G}$, then $g \notin R$ and $g$ does not preserve $R$ as it maps the identity to $g$. Therefore $g \notin \Aut(\Inv(\langle F\rangle_G))$ and the result follows.

[Part $(ii)$.]

[$\langle F\rangle_{DSM} \subseteq \shE(\Inv(F))$.] By induction. One may easily see that if $f,g \in \shE(\Inv(F))$ then $f \circ g \in \shE(\Inv(F))$. Similarly for sub-multipermutations and the identity.

[$\shE(\Inv(F)) \subseteq \langle F\rangle_{DSM}$.] Let $|D|=n$. One may easily see that $\Inv(F)=\Inv(\langle F\rangle_{DSM})$. Let $R$ be the $n^2$-ary relation that lists the shes of $\langle F\rangle_{DSM}$ in the following manner. Consider the $n^2$ positions enumerated in $n$-ary, i.e. by $(i,j)$ s.t. $i,j \in [n]$. Each she $f$ gives rise to many tuples in which the positions $(i,1)$,\ldots, $(i,n)$ are occupied in all possible ways by the elements of $f(i)$. Thus, $f_0:=\she{12}{2}{3}$ generates the following eight tuples
\[
\begin{array}{c}
(1,1,1,2,2,2,3,3,3) \\
(1,1,2,2,2,2,3,3,3) \\
(1,2,1,2,2,2,3,3,3) \\
(1,2,2,2,2,2,3,3,3) \\
(2,1,1,2,2,2,3,3,3) \\
(2,1,2,2,2,2,3,3,3) \\
(2,2,1,2,2,2,3,3,3) \\
(2,2,2,2,2,2,3,3,3) \\
\end{array}
\]
Let $p_{i,j} \in [n]$ be the element at position $(i,j)$. We describe as a \emph{full coding} of $f$ any such tuple s.t., for all $i$, $\{p_{i,1},\ldots,p_{i,|D|}\}=f(i)$. In our example, all tuples except the first and last are full codings of $f_0$. Note that $R$ is preserved by $\langle F\rangle_{DSM}$.
We will prove $\shE(\Inv(\langle F\rangle_G)) \subseteq \langle F\rangle_{DSM}$ by contraposition. If $g$ is a shop not in $\langle F\rangle_{DSM}$, then $g$ does not appear fully coded in $R$ and $g$ does not preserve $R$ as it maps the identity to all tuples that are full codings of $g$. Therefore $g \notin \shE(\Inv(\langle F\rangle_{DSM}))$ and the result follows.
\end{proof}

\vspace{-0.5cm}

\subsection{Lattice isomorphism}

Consider sets of relations $\Gamma$ on the normalised domain $D=[n]$, closed under \mylogic-definability (such sets may be seen as countable signature structures $\mathcal{D}$). Let $\mathscr{R}_n$ be the lattice of such sets ordered by inclusion. Let the lattice $\mathscr{F}_n$ be of DSMs on the set $[n]$, again ordered by inclusion.
\begin{corollary}
\label{cor:galois}
The lattices $\mathscr{R}_n$ and $\mathscr{F}_n$ are isomorphic and the operators $\Inv$ and $\shE$ induce isomorphisms between them.
\end{corollary}
\begin{proof}
From the second parts of Theorems~\ref{thm:galois-connection-inf} and \ref{thm:functional-galois}.
\end{proof}
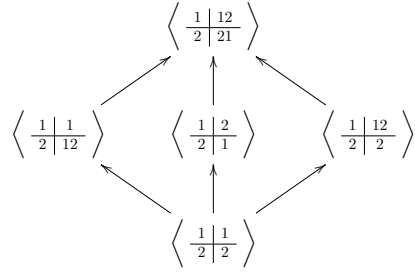
\begin{figure}
\[
\resizebox{!}{.4cm}{
\xymatrix{
& \mbox{$\left\langle \shee{12}{21} \right\rangle$} & \\
\mbox{$\left\langle \shee{1}{12} \right\rangle$} \ar[ur] & \left\langle \shee{2}{1} \right\rangle \ar[u] & \mbox{$\left\langle \shee{12}{2} \right\rangle$} \ar[ul] \\
& \left\langle \shee{1}{2} \right\rangle \ar[ul] \ar[u] \ar[ur] & \\
}
}
\]
\caption{The lattice $\mathscr{F}_2$.}
\end{figure}
The permutation subgroups form a lattice under inclusion whose minimal element contains just the identity and whose maximal element is the symmetric group $S_{|B|}$. As per Theorem~\ref{thm:she-reduction}, this lattice classifies the complexities of $\posFO(\mathcal{B})$ (again there is an isomorphism between this lattice and sets of relations closed under positive fo-definability). In the lattice of DSMs, $\mathscr{F}_n$, the minimal element still contains just $id_B$, but the maximal element contains all multipermutations. However, the lattice of permutation subgroups always appears as a sub-lattice within the lattice of DSMs. In the case of $\mathscr{F}_2$, Figure~1, we have $5$ DSMs, two of which are the subgroups of $S_2$. In the case of $\mathscr{F}_3$, Figure~2, we have $115$ DSMs, only six of which are the subgroups of $S_3$ -- so the lattice complexity jumps very quickly.
\begin{figure}
\begin{center}
\resizebox{!}{8cm}{
\includegraphics[angle=90]{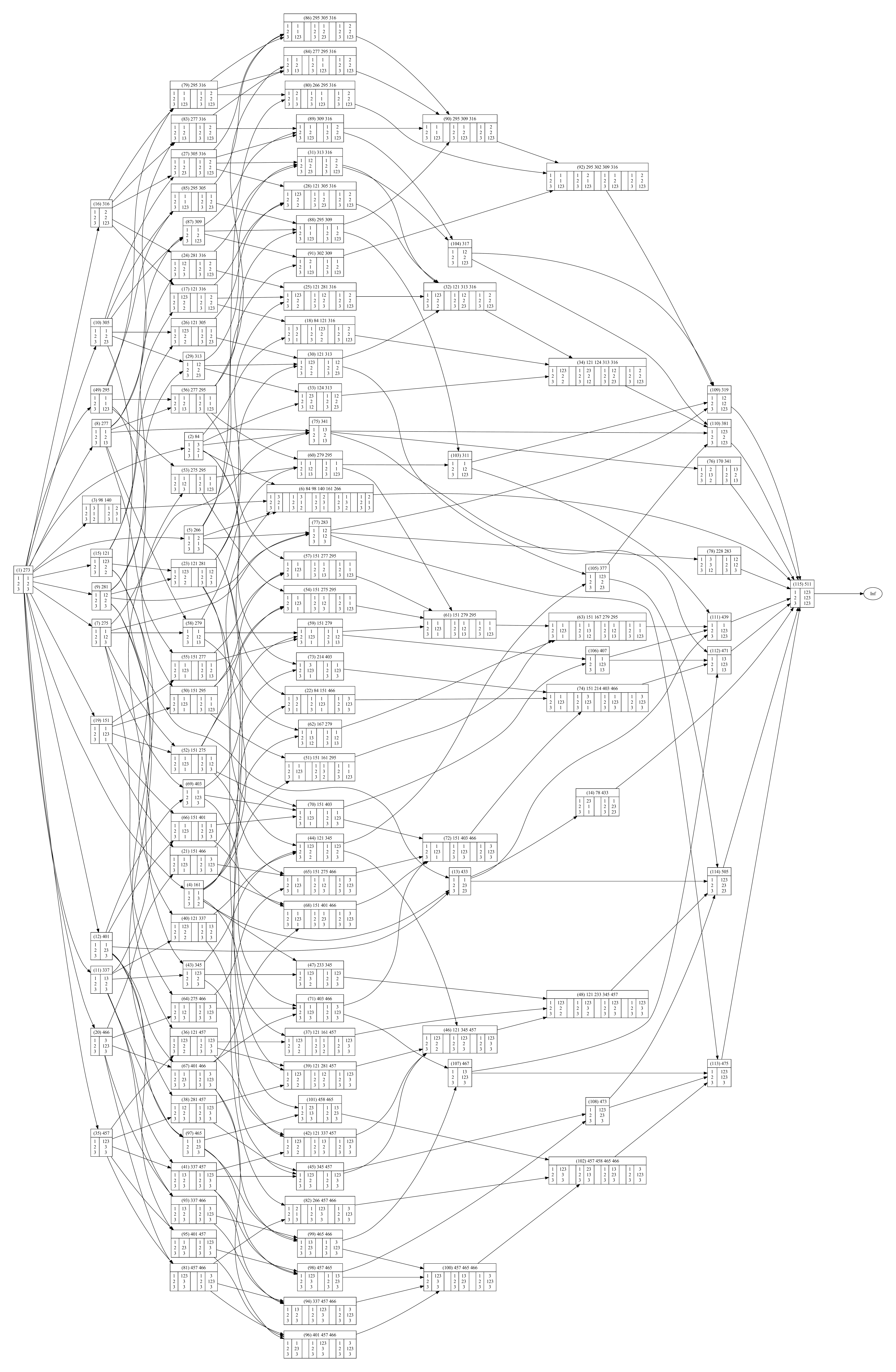}
}
\end{center}
\caption{The lattice $\mathscr{F}_3$. At the bottom is the DSM containing only the identity, at the top is the DSM containing all multipermutations. The authors are grateful to Jos Martin for calculating and drawing $\mathscr{F}_3$. The circular node at the top of the diagram is superfluous to the lattice, being a figment of the imagination of the graph drawing package.}
\end{figure}


\section{The structure of $\mathscr{F}_n$}
\label{sec:structure}
\subsection{Blurred permutation subgroups and symmetric multipermutations}

Many of the DSMs of $\mathscr{F}_n$ are reminiscent of subgroups of the symmetric group $S_{m}$ for some $m \leq n$.
We say that a multipermutation $f$ on the domain $[n]$ is a \emph{blurred permutation} if it may be built from (the multipermutation associated with) the permutation $g$ on the domain $[m]$ ($m \leq n$) in the following manner:
\begin{itemize}
\item[(*)] Let $P_1,\ldots,P_{m}$ be a partition of $[n]$ s.t. $i \in P_i$ for $i \in [m]$ and each $P_i$ may be listed $d_{i,1},\ldots,d_{i,l_i}$. Set $f(d_{i,1})=\ldots =f(d_{i,l_i})=P_{g(i)}$.
\end{itemize}

We note that if $f$ is a blurred permutation obtained as above with partition  $P_1,\ldots,P_{m}$  and permutation $g$,  then $f^{-1}$ is also a blurred permutation obtained with the same partition  $P_1,\ldots,P_{m}$  and permutation $g^{-1}$. This is easy to see if we think of $f^{-1}(d_{i,j})$ as the set of elements that get mapped to $d_{i,j}$ under $f$.

We say that a DSM  $N$ over domain $[n]$ is a \emph{blurred permutation subgroup} (BPS) if one may build it from (the DSM associated with) a subgroup $M$ of $S_{m}$, $m \leq n$ by replacing each permutation $g \in M$ by the blurred permutation $f$ created as in $(*)$, and then taking the closure under sub-multipermutations. We do not feel it necessary to elaborate on this construction save for the example that the group $M:=\langle \shee{2}{1}\rangle$
\begin{itemize}
\item becomes the BPS $N:=\langle \shefour{234}{1}{1}{1}\rangle$ when $P_1:=\{1\}$ and $P_2:=\{2,3,4\}$, and
\item becomes the BPS $N:=\langle \shefour{24}{13}{24}{13}\rangle$ when $P_1:=\{1,3\}$ and $P_2:=\{2,4\}$,
\end{itemize}
and the permutation $\she{1}{3}{2}$ becomes the blurred permutation $\shefour{1}{34}{2}{2}$ when $P_1:=\{1\}$, $P_2:=\{2\}$ and $P_3:=\{3,4\}$. From a blurred permutation (respectively, BPS) one may read the partitions $P_1,\ldots,P_{m}$, for some suitably chosen permutation (respectively, group) on domain $[m]$. A \emph{blurred symmetric group} is a BPS built in the manner described from a symmetric group.

With an arbitrary multipermutation $f$ on $D$, we may associate the digraph $\mathscr{G}_f$ on $D$ in which there is an edge $(x,y)$ if $f(x) \ni y$. The condition of totality ensures $\mathscr{G}_f$ has no sinks and the condition of surjectivity ensures $\mathscr{G}_f$ has no sources. $f$ contains the identity as a sub-multipermutation iff $\mathscr{G}_f$ is reflexive. There is an edge from $a$ to some $y$ in $\mathscr{G}_g$ and an edge from $y$ to $b$ in $\mathscr{G}_f$ iff there is an edge from $a$ to $b$ in $\mathscr{G}_{f \circ g}$. In this fashion, it is easy to verify that there is a directed path of length $n$ from $a$ to $b$ in $\mathscr{G}_f$ iff $b \in f^n(a)$.

Recall that  a multipermutation  (or binary relation) $f$ is  \emph{symmetric} if, for all $a$ and $b$, we have $a \in f(b)$ iff $b \in f(a)$, and it is {\it reflexive} if $a\in f(a)$ for all $a\in D$. Examples of symmetric multipermutations are $id_D$ and $\she{12}{12}{3}$. It not hard to see that $f$ is symmetric iff $f=f^{-1}$ iff $\mathscr{G}_f$ is undirected.

\begin{lemma}\label{diagonal}
The symmetric blurred permutations are the ones built in the manner $(*)$ from an identity multipermutation.  
\end{lemma}
\begin{proof}
Let $f$ be a symmetric blurred permutation obtained as in $(*)$ from the permutation $\sigma$ with partition $P_1, \ldots, P_m$. Recall that $f^{-1}$ is obtained from $\sigma^{-1}$ with the same partition.  We have, for any $i=1,\ldots, m$,  $f(P_i)=P_{\sigma(i)}$, and since $f$ is symmetric we also have  $f^{-1}(P_i)=P_{\sigma(i)}=P_{\sigma^{-1}(i)}$, so that $\sigma=\sigma^{-1}$. Thus $\sigma$ is the identity permutation.

\end{proof}

We note that not all symmetric multipermutations are blurred permutations, for example  $\she{12}{1}{3}$ is a symmetric multipermutation that is not a blurred permutation.

\begin{lemma}
\label{symref}
For all multipermutations $f$, $f \circ f^{-1}$ and $f^{-1} \circ f$ are symmetric and reflexive multipermutations.
\end{lemma}
\begin{proof}
It is sufficient to prove that $f \circ f^{-1}$ is symmetric. Let $a \in f \circ f^{-1}(b)$. Then there exists $y$ s.t. $b \in f^{-1}(y)$ and $y \in f(a)$. Thus, $y \in f(b)$ and $a \in f^{-1}(y)$, i.e. $b \in f \circ f^{-1}(a)$. The fact that they are reflexive is easy to see.
\end{proof}
\noindent Let $f$ and $g$ be symmetric multipermutations on the domain $D$. The minimal symmetric multipermutation $h$ containing both $f$ and $g$ as a sub-multipermutation is said to be the \emph{join} of $f$ and $g$. The \emph{union} $(f \cup g)$ of $f$ and $g$ is the multipermutation given by, for all $x \in D$, $(f \cup g)(x):=f(x) \cup g(x)$.

\begin{lemma}
\label{lem:join}
Let $f$ and $g$ be symmetric and  reflexive multipermutations on the domain $[n]$. The join of $f$ and $g$ is $(f \cup g)^n=(f \circ g)^n=(g \circ f)^n$. \end{lemma}
\begin{proof}
Consider the union $f \cup g$. Since  $f$ and $g$ are both reflexive, we can see that $f \cup g \subseteq f \circ g \subseteq (f \cup g)^2$ and $f \cup g \subseteq g \circ f \subseteq (f \cup g)^2$. The join $h$ of $f$ and $g$ contains exactly those $a \in h(b)$ and $b \in h(a)$ for which there is a path in $f \cup g$ from $a$ to $b$ (and $b$ to $a$). By reflexivity of $f \cup g$ this is equivalent to there being an $n$-path between $a$ and $b$ in $(f \cup g)$, which is equivalent to there being an edge in $(f \cup g)^n$. Noting $(f \cup g)^n=(f \cup g)^{n+1}$, the result follows.   
\end{proof}

\begin{lemma}\label{maximal}
In each DSM there is a unique maximal  reflexive and symmetric multipermutation $g$. Furthermore $g$ is a blurred permutation.
\end{lemma}
\begin{proof}
Since each DSM contains the $id_n$ we know that all DSM contain ay least one reflexive and symmetric multipermutation. Suppose, for a contradiction that there exist two maximal reflexive and symmetric multipermutations, $f, g$,  on a DSM $M$. Then, by Lemma~\ref{lem:join} the join of $f, g$ also belongs to $M$, which contradicts the maximility of $f$ and $g$. Thus in each DSM there exists a unique maximal reflexive and symmetric multipermutation. We now need to show that it is a blurred permutation.  Since for any multipermutation $f$ we have $f\subseteq f^2$, given $g$ the maximal reflexive and symmetric multipermutation on a DSM $M$ with domain $[n]$, we have $g^n=g$ (since $g$ is maximal). To then see that $g$ is a blurred permutation we can either think of the graph $\mathscr{G}_g$ that is the union  of disjoint reflexive cliques, or we can use Theorem \ref{difunctional}, for $g^{-1}=g$ and $g=g^n$.
\end{proof}

 Let $g$ be a blurred permutation on the domain $[n]$ with associated partition $P_1,\ldots,P_{m}$. We say that a multipermutation $f$ \emph{respects} $g$ if \textbf{neither}
\begin{itemize}
\item[$(i)$] exist $a,b$ and $c,d$ s.t. $a,b$ are in the same partition $P_i$ and $c,d$ are in distinct partitions $P_j,P_k$, respectively, and $c \in f(a)$ and $d \in f(b)$, \textbf{nor}
\item[$(ii)$] exist $a,b$ and $c,d$ s.t. $a,b$ are in distinct partitions $P_j,P_k$, respectively, and $c,d$ are in the same partition $P_i$ and $c \in f(a)$ and $d \in f(b)$.
\end{itemize}
\begin{lemma}
\label{lem:respect}
If the multipermutation $f$ does not respect the blurred permutation $g$, then either $(f \circ g) \circ (g^{-1} \circ f^{-1})$ or $(f^{-1} \circ g^{-1}) \circ (g \circ f)$ is a reflexive and symmetric multipermutation that is not a sub-multipermutation of $g$.
\end{lemma}
\begin{proof}
If $f$ does not respect $g$ because of Item $(i)$ above, noting that $g(a)=g(b)$,  then $h:=(f \circ g)\circ (f \circ g)^{-1}$ satisfies $h(c) \supseteq \{c,d\}$. So $h\not\subseteq g$, and the result follows from Lemma \ref{symref} as $(f \circ g)\circ (f \circ g)^{-1}=$ $(f \circ g)\circ (g^{-1} \circ f^{-1})$.

If $f$ does not respect $g$ because of Item $(ii)$ above, then $f^{-1}$ does not respect $g$ because of Item $(i)$ above. The result follows.
\end{proof}
\begin{lemma}
\label{lem:5}
If $g$ is the maximal reflexive and  symmetric multipermutation  in a DSM $M$ and $f$ is a blurred permutation that respects $g$, then $f \in M$ iff there exists $f' \subseteq f$ s.t. $f' \in M$.
\end{lemma}
\begin{proof}
The forward direction is trivial since $M$ is closed under sub-multipermutations. Assume now that there exists $f' \subseteq f$ s.t. $f' \in M$. Since $f$ is a blurred permutation that respects $g$ we can see that the partition of $f$ must be a refinement of the partition of $g$, i.e. if $P_1, \ldots, P_m$ is the partition of  $f$ and $Q_1, \ldots, Q_k$ is the partition of $g$ we have  for each $i=1, \ldots, m$, $P_i\subseteq Q_j$ for some $j=1, \ldots, k$.   Now, applying the multipermutations right to left, we have for any $a\in [n]$ \ $g\circ f'(a)=g(T)$ with $f'(a)=T\subseteq f(a)= P_i \subseteq Q_j$ for some $i=1, \ldots, m$ and $j=1, \ldots, k$, with $g(a)=g(Q_j)$. By Lemmas \ref{diagonal} and \ref{maximal}  it follows that $g$ is a blurred permutation obtained from the identity permutation, so $g(Q_j)=Q_j$. Then $f(a)= P_i \subseteq Q_j= g(f'(a))=g(T)=g(Q_j)$. Hence $f\subseteq g\circ f'$, so $f \in M$.
\end{proof}
 
\subsection{Automorphisms of $\mathscr{F}_n$}

The lattice $\mathscr{F}_n$ has a collection of very obvious automorphisms corresponding to the permutations of $S_n$, in which one transforms a DSM $M$ to $M'$ by the uniform relabelling of the elements of the domain according to some permutation. We will not dwell on these automorphisms other than to give the example that $M:=\langle \she{12}{2}{3} \rangle$ maps to $M':=\langle \she{1}{2}{23} \rangle$ under the permutation $\{1 \mapsto 3, 2 \mapsto 2, 3 \mapsto 1\}$.

There is another, more interesting, automorphism of $\mathscr{F}_n$, which we will call the \emph{inverse} automorphism. We do not close our DSMs under inverse because they were defined in order that the given Galois connections held. It is not hard to verify that if $M$ is a DSM, then $\{ f^{-1} : f \in M \}$ is also a DSM, which we call the \emph{inverse} and denote $M^{-1}$. It is also easy to see that $f=(f^{-1})^{-1}$ and $M=(M^{-1})^{-1}$, from where it follows that inverse is an automorphism of $\mathscr{F}_n$.

\subsection{Properties of inverse}

Call a structure $\mathcal{B}$ \emph{she-complementative} if $\shE(\mathcal{B})=\shE(\mathcal{B})^{-1}$.
Note that, if $F$ is a DSM, then so is $F^{-1}$. In fact, this algebraic duality resonates with the de Morgan duality of $\exists$ and $\forall$, and the complexity-theoretic duality of $\NP$ and $\coNP$ \cite{LICS2009}. 
 
\begin{lemma}
\label{lem:inverse}
For all $\mathcal{B}$, $\shE(\mathcal{B})=\shE(\overline{\mathcal{B}})^{-1}$.
\end{lemma}
\begin{proof}
It follows from the definition of she that $f$ is a she of $\mathcal{B}$ iff $f^{-1}$ is a she of $\overline{\mathcal{B}}$.
\end{proof}
\noindent We are now in a position to derive the following classification theorem.
\begin{theorem}
\label{thm:bps-inverse}
A DSM $N$ is a BPS iff $N=N^{-1}$.
\end{theorem}
\begin{proof}
It is straightforward to see that a BPS $N$ is s.t. $N=N^{-1}$. Specifically, if $f \in N$ then $f$ is derived  from a multipermutation $g$ of a permutation in a group $M$. The inverse $f^{-1}$ may be derived in the same manner from the inverse $g^{-1}$ of $g$.

Now suppose $N$ is s.t. $N=N^{-1}$. Let $g$ be the maximal symmetric multipermutation in $N$. Let $P_1,\ldots,P_{m}$ be the associated partitions of $g$ in the manner previously discussed. Let $M$ be the blurred symmetric group formed from $S_{m}$ by the partitions $P_1,\ldots,P_{m}$. We claim $N \subseteq M$. This follows from Lemmas~\ref{lem:join} and \ref{lem:respect}, since, if $N\notsubseteq M$, then some multipermutation $f \in N$ fails to respect $g$, contradicting the maximality of the symmetric $g$.

This shows that all multipermutations $f \in N$ can be extended to a blurred permutation with partitions $P_1,\ldots ,P_m$. Let $f$ be a multipermutation in $N$ and $f'$ be a blurred permutation with partitions $P_1,\ldots,P_m$  such that $f\subseteq f'$.   By Lemma \ref{lem:5}  $f' \in N$. Thus $N$ is generated by precisely the blurred permutations $f'$, with partitions $P_1,\ldots ,P_m$, that contain the multipermutations  $f \in N$. Thus we see that $N$ is a BPS. 
\end{proof}

We may now give a complexity classification for she-complementative structures based on the following result of \cite{LICS2009}.
\begin{lemma}[\cite{LICS2009}]
\label{lem:bps}
If $\shE(\mathcal{B})$ is a BPS derived from $S_m$, for $m \geq 2$, then $\{\exists, \forall, \wedge,\vee \}$ $\mbox{-}\mathsf{FO}(\mathcal{B})$ is \Pspace-complete.
\end{lemma}
\begin{corollary}
\label{cor:dichotomy}
If $\mathcal{B}$ is she-complementative then $\mylogic(\mathcal{B})$ is either in \Logspace\ or is \Pspace-complete.
\end{corollary}
\begin{proof}
We know from Theorem~\ref{thm:bps-inverse} that $\shE(\mathcal{B})$ is a BPS. If it is a BPS formed from the trivial group $S_1$, then $\shE(\mathcal{B})$ contains all multipermutations. It is easy to see that $\mylogic(\mathcal{B})$ is in \Logspace\ (indeed one may evaluate the quantified variables in an instance arbitrarily - for more details see \cite{LICS2009}). If $\shE(\mathcal{B})$ is a BPS formed from $S_m$, with $m\geq 2$, then $\mylogic(\mathcal{B})$ is \Pspace-complete by Lemma~\ref{lem:bps}.
\end{proof}

\section{The structure of $\mathscr{M}_n$}
\label{sec:Mn}


Let $\mathscr{M}_n$ denote the monoid of all multipermutations on $X=[n]=\{1, 2, \ldots, n\}$ with the binary  operation  of composition of relations: given $\rho, \tau\in \mathscr{M}_n$ 
$$\rho  \circ \tau=\{(x,z) : \ \exists y \in X \ {\rm s.t.}\ (x,y)\in \rho \ {\rm and}\ (y,z)\in \tau\}.  $$
In this section we study the structure of  $(\mathscr{M}_n,  \circ )$  from a semigroup point of view.

Several well studied monoids are associated with  $\mathscr{M}_n$:
the symmetric group 
$S_n$ is a submonoid of $\mathscr{M}_n$; the Hall monoid $H_n$ (where every relations contains a permutation, see for example \cite{ki1974semigroup}) is also a submonoid of $\mathscr{M}_n$; while the transformation semigroup  $T_n$ (all maps from $X$ to $X$) is not contained in $\mathscr{M}_n$ since not all transformations are surjective. Let $\mathcal{B}_n$ be the semigroup of all $n\times n$ Boolean matrices, i.e. with entries from $\{0,1\}$ and the usual matrix multiplication with the assumption that $1+1=1$. This semigroup is isomorphic, in a natural way, to the semigroup of all binary relations on $[n]$, where the operation is the usual composition of relations, and $\mathscr{M}_n$ is (isomorphic to) a submonoid of $\mathcal{B}_n$. We can then think of $\mathscr{M}_n$ as the submonoid of $\mathcal{B}_n$ composed of all boolean matrices with at least one 1 in every  row and every column. We will slightly abuse notation and use $\mathcal{B}_n$ to denote both this monoid and the monoid of all binary relations on $[n]$.



\subsection{Green's relations}
Green's relations for $T_n$ and $S_n$ are trivial and can be found in any Semigroups textbook, in contrast, 
for the semigroup of all binary relations $\mathcal{B}_n$ no simple, direct,  charaterization is known. They were characterised, possibly among others, by \cite{zaretskii1962regular,zaretskii1963semigroup}
 in terms of lattices; \cite{plemmons1970semigroup} in terms of boolean matrices, and by  
 Adu \cite{adu1986green} using direct composition (and skeletons of the relations).
Similarly, for $\mathscr{M}_n$ a simple characterization has so far eluded us.

We can think of the  rows and columns of   an $n\times n$ Boolean matrix  from $\mathcal{B}_n$ as vectors on $\{0,1\}^n$, so these  can be, naturally,  added and compared  coordinate-wise using Boolean operations.

Let $\alpha \in \mathcal{B}_n$. The {\it row space} of $\alpha$, $V(\alpha)$, is the set of all possible sums of rows of $\alpha$, including the zero vector.  Analogously, the  {\it  column space} of $\alpha$, $W(\alpha)$, is the set of all possible sums of columns of $\alpha$, including the zero vector.

\begin{lemma}[Zaretskii]\label{zaretskii}
For any $\alpha, \beta \in \mathcal{B}_n$ :
\begin{enumerate}
\item $\alpha \mathcal{L} \beta \ \Leftrightarrow \ V(\alpha)=V(\beta)$;
\item $\alpha \mathcal{R} \beta \ \Leftrightarrow \ W(\alpha)=W(\beta)$.
\end{enumerate}
\end{lemma}

Following this characterization given by Zarestkii \cite{zaretskii1963semigroup} for Green's relations in $\mathcal{B}_n$, which can also be found in \cite[Lemma 1.2]{plemmons1970semigroup},  we obtained the following characterization for $\mathscr{M}_n$:

\begin{theorem}
Given $\alpha \in \mathscr{M}_n$ let 
$R(\alpha)$  be the set of rows and 
 $C(\alpha)$ the set of columns, respectively,  of $\alpha$. Define $\langle R(\alpha)\rangle=\{ \rho \in V(\alpha)\backslash \{0 \}: \exists \alpha_j\in R(\alpha): \rho\le \alpha_j\}$  and $\langle C(\alpha)\rangle$ analogously.
For any $\alpha, \beta \in M_n$,  we have 
\begin{enumerate}
\item $\alpha \mathcal{L} \beta \ \Leftrightarrow \ \langle R(\alpha)\rangle=\langle R(\beta) \rangle$,
\item  $\alpha \mathcal{R} \beta \ \Leftrightarrow \ \langle C(\alpha)\rangle =\langle C(\beta)\rangle $,
\item $\alpha \mathcal{H} \beta \ \Leftrightarrow \  \langle R(\alpha)\rangle=\langle R(\beta) \rangle {\it and} \ \langle C(\alpha)\rangle =\langle C(\beta)\rangle$.
\end{enumerate}
\end{theorem}
\begin{proof}
For each $\gamma\in  \mathscr{M}_n$ denote by $\gamma_i$ the $i^{th}$ row of $\gamma$ (when $\gamma$ is in matrix form).

Let $\alpha, \beta\in \mathscr{M}_n$ be such that $\langle R(\alpha)\rangle =\langle R(\beta)\rangle $. 
Since $R(\alpha)\subseteq \langle R(\alpha)\rangle$,
  for all $ i\in[ n]$ there exist  $j_{1}, \ldots, j_{l}\in [ n]$ such that  $\alpha_i= \beta_{j_1}+\cdots+\beta_{j_l}$, we can then define $\rho \in \mathscr{M}_n$ by the rule $\rho_i$ has a $1$ in exactly all places $j_1, \ldots, j_l$.  It is then easy to see that $\rho \beta =\alpha$. In a similar way we can define $\delta\in \mathscr{M}_n$ such that $\delta \alpha =\beta$. It follows that  $\alpha \mathcal{L} \beta$.

Let us assume now that $\alpha, \beta \in \mathscr{M}_n$ are such that $\alpha \mathcal{L} \beta$. By Lemma \ref{zaretskii} we know  that $V(\alpha)=V(\beta)$, so, for all $i\in[n]$,  $\alpha_i\in V(\beta)$ which implies that $\alpha_i=\beta_{j_1}+\cdots+ \beta_{j_l}$ for some $\beta_{j_1},\ldots, \beta_{j_l}\in R(\beta)$, since $\alpha_i\neq 0$. Since there exists $\rho\in \mathscr{M}_n$, so $i$ appears in the image of $\rho$ (i.e. $\rho$ has at least one $1$ in column $i$),  such that $\rho\alpha =\beta$, we have that  $\alpha_i\le \beta_j$ for some $j=1, \ldots, n$. Hence $\alpha_i\in \langle R(\beta)\rangle$, and so $R(\alpha)\subseteq \langle R(\beta)\rangle$. It the follows, by definition of $\langle R(\alpha)\rangle$, that $\langle R(\alpha)\rangle \subseteq \langle R(\beta)\rangle$.  Analougously, w ecan show that $\langle R(\beta)\rangle \subseteq \langle R(\alpha)\rangle$. Thus $\langle R(\alpha)\rangle = \langle R(\beta)\rangle$.

In a similar way, reversing rows and columns, we can show that $\alpha \mathcal{R} \beta$ if and only if  $\langle C(\alpha)\rangle =\langle C(\beta)\rangle $. And, as a consequence of both these facts we have that $\alpha \mathcal{H} \beta$ if and only if  $\langle R(\alpha)\rangle=\langle R(\beta) \rangle$ and  $\langle C(\alpha)\rangle =\langle C(\beta)\rangle$.

\end{proof}






  



\begin{example}
We have 

$\she{1}{1}{1,2,3} \ \ \ \mathcal{L} \ \ \ \ \she{1}{1,2,3}{1,2,3} \ \ \ \  \mathcal{R} \ \ \ \ \ \she{1,2}{1,2,3}{1,2,3}$

\end{example}

\begin{example}

The following multipermutations  are $\mathcal{L}$ related in $\mathcal{B}_n$ but not in $\mathscr{M}_n$

$\she{1}{2,3}{1,2,3}$ \ \ \ and \ \ \ $\she{1}{2,3}{2,3}$

\end{example}

\begin{example}
The following multipermutations are $\mathcal{L}$ related but do not have the same set of rows
$$\shefive{1,3}{1}{3}{2,4}{1,3,5} \ \ \ \mathcal{L} \ \ \ \ \shefive{1}{3}{2,4}{2,4}{1,3,5}$$

\end{example}

\subsection{Regular elements}

An element $a$ of a semigroup $S$ is called {\it regular} if there exists $x\in S$ s.t. $a=axa$, $x$ is called an inverse of $a$. 
Schein \cite{schein1976regular} gave us a way of checking if a binary relation is regular in the semigroup of all binary relations.

\begin{lemma}[\cite{schein1976regular}]
Let $\rho\in \mathcal{B}_n$ be a binary relation. Then 
$\rho$ is regular (in $\mathcal{B}_n$) iff $\rho \subseteq \rho \circ (\rho^{-1}\circ \rho^c\circ \rho^{-1})^c\circ \rho$.
\end{lemma}
Here $\rho^{-1}$ is the inverse relation, as defined earlier on for multipermutations $\rho^{-1}=\{(y,x) : (x,y)\in \rho\}$, and $\rho^c$ is the complement relation  $\rho^c=\{(x,y)\in X\times X: (x,y)\notin \rho\}$. Since we are here using the word inverse for distinct things we will use just inverse for semigroup inverse and will call inverse relation to $\rho^{-1}$ to avoid confusion.
In the same paper he showed that the relation $$(\rho^{-1}\circ \rho^c \circ \rho^{-1})^c\circ \rho \circ (\rho^{-1}\circ \rho^c \circ \rho^{-1})^c$$
is the greatest (relatively to containment of relations) inverse of $\rho$.

Using Schein's condition we can  check if a multipermutation $\rho$ has an inverse, by checking if this greatest inverse is also a multipermutation, but it is not enough to check the regularity condition presented in the lemma above.

\begin{example} The multipermutation
$\she {1,2}{3}{1,2,3}$ is regular as a binary relation, but not as a multipermutation, since all inverses of it are binary relations that are not multipermutations.

\end{example}

Even though this greatest inverse is computable in polynomial time it is not always simple to check. We adapted an algorithm by Kim \& Roush \cite{kim1978inverses} to compute inverses for multipermutations, in particular em keep the notation used in that article for easier comparison.

Let $V_n$ be the set of all $n$-tuples  of elements of $\{0,1\}$. A subset $W$ of $V_n$ is called a {\it subspace} of $V_n$ if it contains the zero vector and $u+v\in W$ for all $u, v\in W$. The {\it subspace spanned} by $W$ is the smallest subspace that contains $W$, we denote it by $\langle W\rangle$. A vector $v$ is said to be {\it dependent} on $W$ if $v\in \langle W\rangle$. A set $W$ is said to be {\it independent} if for all $v\in W$, $v$ is not dependent on $W\backslash \{v\}$. A subset $S$ is said to be a {\it basis} for a subspace $W$ if $W=\langle S\rangle$ and $S$ is an independent set.

If $v\in V_n$ we denote by $v_i$ the element of $\{0,1\}$ occuring in position $i$ of $v$. For $u, v\in V_n$ we say that $u\leqslant v$ if $u_i=1$ only if $v_i=1$ for all $i=1, \ldots, n$.

\bigskip
{\bf Computing regular multipermutations:}

Let $\alpha$ be a binary square matrix (i.e. a multipermutation). Let   $\alpha_{i^*}$  denote the $i^{th}$ row of $\alpha$.

By the row space of $\alpha$ we mean the subspace spanned by the set of rows of $\alpha$, and denote it by $R(\alpha)$. We denote by $b(\alpha)$ the basis for $R(\alpha)$ and call it the {\it row basis} of $\alpha$.

For each $v \in b(\alpha)$ a vector $u$ with one $1$ is called an {\it identification vector}  of $v$ if and only if:  $u\leqslant w$ holds if and only if   $v\leqslant w$ for $w\in b(\alpha)$. 

Let $I(v)$ denote the set of identification vectores  of a basis vector $v$ of $\alpha$. Finally, set
 $p(t)=\{\inf w\in R(\alpha): t\leq w\}$, where the  infimum  is taken in the lattice $V(\alpha)$.

The algorithm  to find an  inverse multipermutation of  $\alpha$, receives as input $\alpha$, and goes :

\begin{enumerate}
\item find $b(\alpha)$;
\item find $I(v)$ for each $v\in b(\alpha)$; 
\item for each $v\in b(\alpha)$, choose a  specific identification vector $u\in I(v)$;
\item for each such chosen  $u$, choose a vector $s$  s.t. $s_i=1$  if $\alpha_{i^*}\leqslant v$;
\item choose any vector $t$ with exactly one $1$ entry other than the $u$'s chosen in step 3, and  send $t$ to a vector $b$ s.t. $b_i=1$ only if $\alpha_{i^*}\leqslant p(t)$;
\item linearly order the set of vectores with only one 1 in such a way that the mapping $i\mapsto (\delta_{i1},\delta_{i2},  \ldots, \delta_{in})$ is an order isomorphism. Write the vectors $s$ and $b$ in the order of the $u$'s and $t$. \end{enumerate}

The resulting matrix will be an inverse of $\alpha$, that will also be a multipermutation when the conditions presented in Theorem \ref{regular} are satisfied.
We note that this algorithm differs from\cite[Section 5 I]{kim1978inverses} in Step 5 since for the case of multipermutations all columns of the matrix have a  $1$.

\begin{example}

Let $\alpha= \she{2}{2,3}{1}$, following the steps above we get:
\begin{enumerate}
\item $b(\alpha)=\{(0 \ 1 \ 0), (0 \ 1 \ 1), (1 \ 0\ 0)\}$ which we note is equal to $R(\alpha)$;
\item $I((0 \ 1 \ 0))=\{ (0 \ 1 \ 0)\}; I((0 \ 1 \ 1))=\{(0 \ 0\ 1)\}; I((1 \ 0\ 0))=\{(1 \ 0\ 0)\}$;
\item the $u$'s are clearly defined, no choice needs to be made;
\item For $u=(0 \ 1 \ 0)$, $s$ must be $(1 \ 0 \ 0)$; for $u=(0 \ 0\ 1)$ the vector $s$ can be $(1\ 1\ 0) $ or $(0\ 1\ 0)$; for $u=(1 \ 0\ 0)$, $s$ must be $(0\ 0 \ 1)$;
\item no choice for $t$;
\item there are two inverses of $\alpha$ in $\mathscr{M}_n$, they are
\end{enumerate}
 \[
\begin{bmatrix}
    0&0&1\\
    1&0&0\\
    1&1&0
\end{bmatrix}
{\rm and}
\begin{bmatrix}
    0&0&1\\
    1&0&0\\
    0&1&0
\end{bmatrix} \].

\end{example}

\begin{lemma}
Let $\alpha\in \mathscr{M}_n$ be arbitrary.
If $b(\alpha)\neq R(\alpha)$ then $\alpha$ has no inverse in $\mathscr{M}_n$.
\end{lemma}
\begin{proof}
Suppose that $\alpha_1\neq b(\alpha)$, so that $\alpha_1$ can be written as the sum of other rows of  $\alpha$. We show that all inverses of $\alpha$ in $\mathcal{B}_n$ will have only zeros in the first column.
Following Kim and Roush algorithm, we know that $\alpha_1\neq v$ for all $v\in b(\alpha)$, so no vector $s$, produced by the algorithm, will have a $1$ in the first component. Since $t$ contains a unique $1$ we know that $p(t)\neq \alpha_1$, it follows that $b$ will not contain $1$ in its first component. 
\end{proof}

\begin{lemma}
Let $\alpha\in \mathscr{M}_n$ be arbitrary.
If $p(t)=0$, for any possible $t$,  then $\alpha$ has no inverse in $\mathscr{M}_n$.
\end{lemma}
\begin{proof}
If this is the case, then in  Kim and Roush algorithm $b=0$, since no row of $\alpha$ is the zero row, so all inverses of $\alpha$ will have a zero row, thus they will not be multipermutations.
\end{proof}

\begin{theorem}\label{regular}
A multipermutation $\alpha$ has an inverse, in $\mathscr{M}_n$,  iff $b(\alpha)=R(\alpha)$, $I(v)\neq \emptyset$ for each $v\in b(\alpha)$, and $p(t)\neq 0$ (for some $t$).
\end{theorem}
\begin{proof}
$\Rightarrow$ Follows from \cite[Lemma 1]{kim1978inverses}  and the two previous lemmas.

$\Leftarrow$ Assume that $b(\alpha)=R(\alpha)$, $I(v)\neq \emptyset$ for each $v\in r(\alpha)$, and $p(t)\neq 0$ (for some $t$). By \cite[Lemma 1]{kim1978inverses} we know that $\alpha$ has an inverse in $\mathcal{B}_n$. Since $p(t)\neq 0$ for some $t$ we know that $\alpha$ will have an inverse with no zero row. We now just need to show that one of the inverses with no zero row also has  no zero column.

We know, by assumption, that $\alpha_1^*$ (the first row of $\alpha$) belongs to $b(\alpha)$, hence, when we are at Step 3 of the algorithm and $v=\alpha_1^*$ we will choose $s$ such that $s_1=1$. Thus the first column of the inverse of $\alpha$ will be non zero. Since all rows of $\alpha$ belong to $b(\alpha)$ it follows that we can choose an inverse with non zero columns. Thus $\alpha$ has an inverse in $\mathscr{M}_n$.
\end{proof}

\begin{example}
$\she{1,2,3}{2,3}{1}$
has inverses in $\mathcal{B}_n$ but not in $\mathscr{M}_n$. From Kim and Roush algorithm we can check that all its inverses are
$$\she{3}{2}{}, \ \ \ \she{3}{2}{2}, \ \ {\rm and} \ \ \she{3}{}{2},$$
none being a multipermutation. this follows from the fact that its row basis does not include all rows.
\end{example}

\subsection{Generators}

It is known that, unlike $S_n$ that is generated by two elements,  $\mathcal{B}_n$ does not admit a polynomial (on $n$) generating set. This was described by Devadze \cite{devadze1968generating} and more recently proved by 
Konieczny \cite{konieczny2011proof}.

In this section we show that $\mathscr{M}_n$ also does not admit a polynomial generating set, and does indeed need more elements to be generated than $\mathcal{B}_n$.

Using Devadze's set of generators, and the proof provided by Konieczny, we show that any set of generators of $\mathscr{M}_n$ must include the two permutations that generate $S_n$ and a set of representatives of the prime $\mathcal{D}$-classes of $M_n$.
To obtain a generating set we add a few more multipermutations to the set mentioned above.

Let $\alpha, \beta, \gamma \in \mathcal{B}_n$, the monoid of binary relations. We say that $\alpha$ is {\it prime} if  it is not a permutation and if $\alpha=\beta \gamma$ implies that either $\beta$ or $\gamma$ are a permutation.

De Caen and Gregory\cite{de1981primes} showed that   if $\alpha \in \mathcal{B}_n$ is prime then no column of $\alpha$ can contain another, and no row of $\alpha$ can contain another row. In particular if $\alpha \in \mathcal{B}_n$ is prime then $\alpha$ has no zero row or column and no row or column with all entries  equal to $1$.
This means that all prime elements of $\mathcal{B}_n$ are multipermutations.
In the same paper they also show that prime multipermutations are not regular, and if a $\mathcal{D}$-class of $\mathcal{B}_n$ contains a prime relation then all relations in that class are prime. We will call these classes {\it prime $\mathcal{D}$-classes}, and they are $\mathcal{D}$-classes of $\mathscr{M}_n$ that are located just below the group of units $S_n$ in the partial order of $\mathcal{D}$-classes of $\mathscr{M}_n$.  This can also be found in  \cite{konieczny2011proof} without mentioning multipermutations.

 We are now trying to build a generating set for $\mathscr{M}_n$, and it follows from Konieczny's result, adapted to multipermutations,  that any set of generators must contain a set of generators of $S_n$ and a set of representatives of the prime $\mathcal{D}$-classes of $\mathscr{M}_n$. The following is the equivalent of \cite[Lemma 4.2]{konieczny2011proof}.

 \begin{lemma}
  Let $D$ be a prime $\mathcal{D}$-class of $\mathscr{M}_n$ and let $T$ be a set of generators of $\mathscr{M}_n$.
Then $D\cap T\neq \emptyset$.
  \end{lemma}
 \begin{proof}
  Assume, for a contradiction,  that $D \cap T =\emptyset$. Let $$m=\min \{k :\alpha=t_1\cdots t_k \ {\rm  for \ some}\ \alpha \in D \ {\rm and} \ t_1,\ldots,t_k\in T\}.$$
   Choose some  $\alpha \in D$ such that $\alpha=t_1 \cdots t_k$ for some   $t_1,\ldots ,t_m\in T$. 
 Since $\alpha \notin T$ we
have $m\ge 2$. Note that $t_1 \notin S_n$ since otherwise $t_1^{-1}\alpha = t_2\cdots t_m \in D$, which would 
contradict the minimality of $m$. Similarly, $t_m \notin S_n$. Since $t_m$ is a multipermutation we must have that  $|t_m(i)|\ge2$ (or when in matrix form, there is a row of $t_m$ with at least two $1$s) for some $i=1,\ldots, n$, it follows that $|t_2\cdots t_m (j)|\ge 2$ for some $j=1,\ldots, m$ (note that we apply the relations left to right). Hence $ t_2 \cdots t_m \notin S_n$, and since 
 $t_1 \notin S_n$, and $\alpha = t_1(t_2 \cdots t_m)$, which is a contradiction since $\alpha$ is prime. Thus $D \cap  T\neq \emptyset$.
 
 \end{proof}
 
 The number of prime $\mathcal{D}$-classes  grows faster than a polynomial on $n$, so we won't be able to find a minimal generating for $\mathscr{M}_n$ that is polynomial.
 A minimal generating set for it will contain a set of representatives of the prime $\mathcal{D}$-classes,  the two permutations that generate $S_n$, the multipermutation (in matrix form)
$$\pi=\begin{bmatrix}
    1&0&0&0&\cdots &0 & 0\\
    1&1&0&0&\cdots &0 & 0\\
    0&0&1& 0& \cdots&0 &0\\
    
    \ldots\\
    0&0&0&0&\cdots &0& 1
\end{bmatrix}$$
and a few more multipermutations. This will be the subject of future work, and we leave here a few examples that were tested using GAP \cite{GAP4, Mitchell2020aa}:

\begin{example}
The only prime element in $\mathscr{M}_3$ (up to equivalence) is $\she{1,2}{2,3}{1,3}$ \cite[Example 2.3]{de1981primes}. A generating set for $\mathscr{M}_3$ is the given by 
the permutations $(1 \ 2), (1\ 2\ 3)$, the prime multipermutation,  the multipermutation $\she{1}{1,2}{3}$ (called $\pi$ above) and $\she{1}{1}{2,3}$.  

\end{example}

\begin{example}
The prime elements in $\mathscr{M}_4$ (up to equivalence) are \cite[Example 2.5]{de1981primes}
$$\begin{bmatrix}
    0&1&1&1\\
    1&1&0&0\\
    1&0&1& 0\\
    1&0&0&1
\end{bmatrix}\ \ \ 
and 
\ \ \ \begin{bmatrix}
    1&0&0&1\\
    1&1&0&0\\
    0&1&1& 0\\
    0&0&1&1
\end{bmatrix}$$
and they belong to different $\mathcal{D}$-classes.
A generating set for $\mathscr{M}_4$ is given by the  $(1 \ 2), (1\ 2\ 3\ 4)$, the prime multipermutations above, the  multipermutation $\shefour{1}{1,2}{3}{4}$ (called $\pi$ above), together with the multipermutations $\shefour{1}{2}{2}{3,4}, \shefour{1,2}{1,3}{2,3}{4}$.
\end{example}

\subsection{Blurred permutations}

A multipermutation, or more generally binary relation, $\rho$, is called {\it difunctional}  if it satisfies $\rho \circ \rho^{-1}\circ \rho \subseteq \rho$.
Schein \cite{schein1987multigroups}  showed that every inverse semigroup is isomorphic to an appropriate inverse semigroup of full difunctional binary relations (here the operation is not usual composition since the composition of two difunctional binary relations is not necessarily difunctional). In this subsection we relate blurred permutations with difunctional relations.

\begin{lemma}
Every blurred permutation is  difunctional. 
\end{lemma}
\begin{proof}
Let $\rho$ be a blurred permutation obtained from the permutation $g$ with partition $P_1, \ldots, P_m$, then $\rho^{-1}$ is obtained from $g^{-1}$ with partition $P_{g(1)}, \ldots, P_{g(m)}$. For any $i=1,\ldots,m$ and any $x\in P_i$ we have $f(x)=f(P_i)=P_{g(i)}$, and  $f\circ f^{-1}\circ f(x)=f\circ f^{-1} (P_{g(i)})=f(P_i)=P_{g(i)}$. Thus $f\circ f^{-1}\circ f=f$, and $f$ is difunctional.
\end{proof}

It is also clear from this  proof that for blurred permutations are all regular and the inverse multipermutation is also an inverse (in the sense of regular element). We also note that in general we compose relations from left to right and the notation above seemed to compose them right to left, we just used this notation for easiness since in this case reading the composition left to right or right to left made no difference.

\begin{theorem}\label{difunctional}
Blurred permutations are the difunctional multipermutations $\rho$ that satisfy $\rho\circ \rho^{-1}=\rho^{-1}\circ \rho$. Hence they can be defined exactly by the rules $\rho\circ \rho^{-1}\circ \rho=\rho$ and $\rho\circ \rho^{-1}=\rho^{-1}\circ \rho$, or equivalently are full total binary relations on $X$ of the form $A_1\times B_1 \cup \cdots \cup A_k\times B_k$, with $\{A_1, \ldots A_k\}= \{B_1,\ldots, B_k\}$ partitions of $X$.

\end{theorem}
\begin{proof}
We can see in \cite{schein1987multigroups} a result attributed to J. Riguet '48,'51 that says that a binary relation is difunctional if and only if it is of the form $A_1\times B_1 \cup \cdots \cup A_k\times B_k$, with $A_1, \ldots, A_k$ all distinct and $B_1, \ldots, B_k$ all distinct. So we can say that a multipermutation  on $[n]$ is difunctional if and only if it is of the form $A_1\times B_1 \cup \cdots \cup A_k\times B_k$, with $\{A_1, \ldots A_k\}, \{B_1,\ldots, B_k\}$ partitions of $[n]$.
We now need to show that  $\{A_1, \ldots A_k\}= \{B_1,\ldots, B_k\}$. 

Let $\rho$ be a blurred permutation. Since it is difunctional it satisfies $\rho\circ \rho^{-1}\circ \rho=\rho$, so we just need to show it satisfies $\rho\circ \rho^{-1}=\rho^{-1}\circ \rho$. Suppose that $\rho$ is obtained from permutation $g$ and partition $A_1, \ldots, A_k$, then  $f^{-1}\circ f (A_i)= f^{-1}(A_{g(i)}) =A_{g^{-1}g(i)}=A_i$ and $f\circ f^{-1}(A_i)= f(A_{g^{-1}(i)})=A_{gg^{-1}(i)}=A_i$. Thus $f\circ f^{-1}=f^{-1}\circ f$. 

From this we can also see that the $A_1\times B_1 \cup \cdots \cup A_k\times B_k$ can be rewritten as $A_1\times A_{g(1)} \cup \cdots \cup A_k\times A_{g(k)}$, so it follows that  $\{A_1, \ldots A_k\}= \{B_1,\ldots, B_k\}$.

We now show the reverse implication. If $\rho= A_1\times B_1 \cup \cdots \cup A_k\times B_k$ with $\{A_1, \ldots A_k\}= \{B_1,\ldots, B_k\}$ we can see that it is a blurred permutation obtained from the permutation that sends $A_i$ to $B_i$. 

If we assume that $\rho$ is a multipermutation that satisfies $\rho=\rho \circ \rho^{-1}\circ \rho$ and $\rho\circ\rho^{-1}=\rho^{-1}\circ \rho$, we know it is difunctional, so $\rho =A_1\times B_1 \cup \cdots \cup A_k\times B_k$ , then 
$\rho^{-1}\circ \rho =B_1\times B_1 \cup \cdots B_k\cdots B_k$ and $\rho\circ \rho^{-1}= A_1\times A_1\cup A_k\times A_k$. It then follows that we must have $\{A_1, \ldots, A_k\}=\{B_1, \ldots, B_k\}$, so $\rho$ is a blurred permutation.

\end{proof}

In other words, blurred permutations are the completely regular difunctional multipermutations. For the definition of completely regular see for example \cite{clifford1941semigroups}.

It is tempting at this stage to compare blurred permutations with Hall's relations, the connection does not seem to be a direct one. It is easy to find a multipermutation that is not a Hall relation, so the best  we can say at this stage is that all symmetric blurred permutations are Hall's relations, as a direct consequence of Lemma \ref{diagonal}.

\bibliographystyle{acm}

\end{document}